\documentclass[11pt]{article}
\usepackage[a4paper,margin=27mm]{geometry}
\usepackage{amsmath,amssymb,amsthm}
\usepackage[T1]{fontenc}
\usepackage{hyperref}
\hypersetup{colorlinks=true,linkcolor=blue,urlcolor=blue,
 pdftitle={Maximum-modulus points on concentric circles},
 pdfauthor={Yunpeng Li}}
\newtheorem{theorem}{Theorem}
\newtheorem{lemma}[theorem]{Lemma}
\newtheorem{proposition}[theorem]{Proposition}
\newtheorem{corollary}[theorem]{Corollary}
\theoremstyle{remark}
\newtheorem{remark}[theorem]{Remark}
\newtheorem{example}[theorem]{Example}
\numberwithin{equation}{section}
\newcommand{\C}{\mathbb C}
\newcommand{\N}{\mathbb N}

\newcommand{\Cs}{\C^{\times}}
\newcommand{\Chat}{\widehat{\C}}
\newcommand{\card}{\operatorname{card}}
\newcommand{\ord}{\operatorname{ord}}
\title{Maximum-modulus points on concentric circles}
\author{Yunpeng Li\\[3pt]
 \small School of Data Science, The Chinese University of Hong Kong, Shenzhen, China\\
 \small\texttt{liyunpeng@cuhk.edu.cn}}
\date{14 October 2026}
\begin{document}
\maketitle

\begin{abstract}
We give a negative answer to the second question of Erd\H{o}s
Problem~\#1117: can the number $\nu_f(r)$ of maximum-modulus points of a
nonzero, non-monomial entire function on $|z|=r$ tend to infinity as $r\to\infty$?
Write $f(z)=z^m g(z)$ with $g(0)\ne0$, and let $k$ be the first positive
exponent with a nonzero coefficient in the Taylor expansion of $g$.
We prove that $\{r>0:\nu_f(r)>2k\}$ is locally finite in $(0,\infty)$,
and that the constant $2k$ is sharp. In particular,
$\liminf_{r\to\infty}\nu_f(r)\le2k<\infty$, without a growth assumption.
The proof combines finite local parametrizations of a
logarithmic-derivative correspondence with analytic continuation and a
local description of stationary points.
\end{abstract}

\section{The problem and the result}\label{sec:introduction}

For an entire function $f$ and $r>0$, let
\[
 M_f(r)=\max_{|z|=r}|f(z)|,\qquad
 \mathcal M_f(r)=\{z\in\C:|z|=r,\ |f(z)|=M_f(r)\}.
\]
The maximum exists by compactness. Throughout, a \emph{non-monomial
entire function} means a nonzero entire function that is not $cz^m$
with $c\ne0$ and $m\in\N=\{0,1,2,\ldots\}$. For such a function,
$\mathcal M_f(r)$ is finite at every positive radius
(Lemma~\ref{lem:circle-finite}), and we write
$\nu_f(r)=\card\mathcal M_f(r)$.

Around 1964, Erd\H{o}s asked whether a non-monomial entire function
could satisfy either of the following conditions
\cite[p.~96]{Piranian1977}:
\begin{equation}\label{eq:erdos-questions}
 \begin{aligned}
 \text{(a)}\quad&\limsup_{r\to\infty}\nu_f(r)=\infty,\\
 \text{(b)}\quad&\liminf_{r\to\infty}\nu_f(r)=\infty.
 \end{aligned}
\end{equation}
These are the two questions in Erd\H{o}s Problem~\#1117
\cite{EP1117}. They were recorded as Problem~2.16 in Hayman's 1967
book \cite{Hayman1967}, and are reproduced in Hayman--Lingham
\cite[Problem~2.16, printed p.~29]{HL} and
Gl\"ucksam--Pardo-Sim\'on \cite[Question~1.1]{GP}.
The second question was later also attributed to Clunie; see
\cite{GP} for the historical discussion.

Herzog--Piranian \cite{HP} answered~(a) affirmatively in 1968.
Pardo-Sim\'on--Sixsmith \cite{PS} subsequently obtained an affirmative
answer with a finite-order entire function. These results give
unbounded counts along selected radii. Question~(b) asks for more:
one function whose count exceeds every fixed integer on all sufficiently
large circles. Gl\"ucksam--Pardo-Sim\'on \cite{GP} obtained an approximate
version of~(b), involving points close to the maximum modulus.

We answer~(b) negatively. In fact, the count is bounded by a constant
determined by the gap between the first two nonzero Taylor terms of $f$, except
at a locally finite set of radii.

\begin{theorem}\label{thm:main}
Let $f$ be a non-monomial entire function. Write
\[
 f(z)=z^m g(z),\qquad
 g(z)=a_0+a_kz^k+O(z^{k+1}),\qquad a_0a_k\ne0,
\]
where $g$ is entire, $m=\ord_0 f$, and $k\ge1$ is the first positive exponent occurring
in the Taylor expansion of $g$. Then
\[
 E_f=\{r>0:\nu_f(r)>2k\}
\]
is locally finite in $(0,\infty)$: it meets every compact subinterval
in finitely many points. The constant $2k$ is sharp.
\end{theorem}

\begin{corollary}\label{cor:negative}
For every $f$ as in Theorem~\ref{thm:main} and every $R>0$, there exists
$r>R$ with $\nu_f(r)\le2k$. Hence
\[
 \liminf_{r\to\infty}\nu_f(r)\le2k<\infty.
\]
In particular, $\nu_f(r)$ cannot tend to infinity as $r\to\infty$.
\end{corollary}

\begin{proof}
The compact interval $[R,R+1]$ contains only finitely many points of
$E_f$. Any radius in $(R,R+1)$ outside that finite set has count at
most $2k$. Taking arbitrarily large $R$ gives the liminf bound.
\end{proof}

\begin{remark}[Formal verification]\label{rem:lean}
The boundedness conclusion of Corollary~\ref{cor:negative} and its
negative answer to the second question of Erd\H{o}s Problem~\#1117
have a Lean~4 formalization.
The source and verification records are available from the author's
website \cite{LiLean}. The local-finiteness theorem and the sharpness
examples are proved analytically here and lie outside that formalization.
\end{remark}

We first encode maximum points as inverse pairs of a
logarithmic-derivative correspondence. Near the origin, the Taylor
expansion bounds the number of pairs with a small coordinate.
The key bound concerns points where the image is locally a graph over
the common value. If this bound failed, the closure of the graph points
with too many inverse pairs would consist locally of full analytic branches,
even at the excluded value and at infinity. Its projection to
the product coordinate would be both open and closed in $\C$, forcing
it to reach the small products where the bound is already known.
This contradiction gives the required fiber bound.

The passage from inverse pairs to circles has two steps. The radial
maximum function supplies a radius with $\nu_f(r)\le2k$ in every
positive interval. A local parametrization of stationary points then
shows that $\nu_f$ is constant on each sufficiently short side of any
fixed radius. These two facts imply local finiteness of $E_f$.

\section{Factorization and finiteness on each circle}\label{sec:finiteness}

Fix $f=z^m g$ and $k$ as in Theorem~\ref{thm:main}. The function $g$
is entire, nonconstant, and nonzero at the origin. Since
$|f(z)|=r^m|g(z)|$ on $|z|=r$, the functions $f$ and $g$ have the same
maximum points on every positive circle.

For an entire or meromorphic function $h$, write
\[
 h^\sharp(z)=\overline{h(\bar z)}.
\]
This reflection preserves holomorphy and meromorphy. The following
product will prove finiteness on circles and later exclude branches
with constant product coordinate. Here $\Cs=\C\setminus\{0\}$.

\begin{lemma}\label{lem:product}
For each $s\in\Cs$, the function
$J_s(z)=g(z)g^\sharp(s/z)$ is holomorphic and nonconstant on $\Cs$.
\end{lemma}

\begin{proof}
Holomorphy follows by composition and multiplication. Suppose that
$J_s$ is a constant $c_0$. As $z\to\infty$, the factor
$g^\sharp(s/z)$ tends to $g^\sharp(0)\ne0$, so
$g(z)=c_0/g^\sharp(s/z)$ is bounded outside a sufficiently large disk.
It is also bounded on the remaining compact disk. Liouville's theorem
would make $g$ constant, a contradiction.
\end{proof}

\begin{lemma}\label{lem:circle-finite}
The set $\mathcal M_f(r)$ is finite for every $r>0$.
\end{lemma}

\begin{proof}
It suffices to consider $g$. If $\mathcal M_g(r)$ were infinite, it
would have an accumulation point on the circle, hence in $\Cs$.
At every maximum point $z$, the identity $r^2/z=\bar z$ gives
\[
 J_{r^2}(z)=g(z)g^\sharp(\bar z)=|g(z)|^2=M_g(r)^2.
\]
The identity theorem would make $J_{r^2}$ constant on the connected
domain $\Cs$, contrary to Lemma~\ref{lem:product}.
\end{proof}

\section{Logarithmic derivatives and small products}\label{sec:small}

Define the meromorphic functions
\[
 A(z)=m+\frac{zg'(z)}{g(z)},\qquad B(z)=A^\sharp(z).
\]
Away from zeros of $f$, one has $A(z)=zf'(z)/f(z)$.
We regard $A$ and $B$ as continuous maps into the Riemann sphere
$\Chat$, with conjugation fixing $\infty$. Both are holomorphic at zero,
and $A(0)=B(0)=m$.

The Taylor expansion of $g$ gives
\[
 A(z)-m=\frac{zg'(z)}{g(z)}
       =\frac{k a_k}{a_0}z^k+O(z^{k+1}).
\]
Thus $A-m$ has a zero of order $k$ at the origin, as does $B-m$.
In particular, $A$ and $B$ are nonconstant.
At a nonzero zero $\alpha$ of $g$ of multiplicity $n_\alpha$, write
$g(z)=(z-\alpha)^{n_\alpha}h(z)$ with $h(\alpha)\ne0$. Then
\[
 A(z)=\frac{n_\alpha\alpha}{z-\alpha}
       +m+n_\alpha+\frac{zh'(z)}{h(z)}.
\]
Consequently, all poles of $A$ and $B$ are simple.

For $s\in\Cs$ and $p\in\Chat$, define the inverse-pair set
\[
 \mathcal F(s,p)=
 \{(z,w)\in(\Cs)^2:zw=s,\ A(z)=B(w)=p\}.
\]
Its cardinality always counts distinct ordered pairs, without
multiplicity. Either coordinate determines the other.
For maximum points, the relevant pairs are $(z,\bar z)$ and their
product is $r^2$. Section~\ref{sec:radii} will choose a circle on which
all maximum points give pairs in one common fiber.

We use a local coordinate fact repeatedly. If a nonconstant
holomorphic function $h$ satisfies $h(\alpha)=p_0$, factor
\[
 h(z)-p_0=(z-\alpha)^d u(z),\qquad u(\alpha)\ne0.
\]
A holomorphic $d$th root of $u$ on a sufficiently small disk gives
the conformal coordinate
$e(z)=(z-\alpha)u(z)^{1/d}$, in which $h(z)=p_0+e(z)^d$.
Each value therefore has at most $d$ distinct inverse roots in that
disk.

Apply this fact to $A$ and $B$ at zero. Choose $\varepsilon>0$ such
that each function has at most $k$ inverse roots in $|z|<\varepsilon$,
for every finite value, and has no poles there. In any
$\mathcal F(s,p)$, at most $k$ pairs have $|z|<\varepsilon$ and at
most $k$ have $|w|<\varepsilon$. Thus at most $2k$ pairs have a small
coordinate. If $0<|s|<\varepsilon^2$, every pair has a small coordinate,
because $zw=s$. Hence
\begin{equation}\label{eq:small}
 0<|s|<\varepsilon^2
 \quad\Longrightarrow\quad
 \card\mathcal F(s,p)\le2k
 \qquad(p\in\Chat).
\end{equation}
In particular, these fibers are finite. For $p=\infty$ there are no
inverse roots in the small disk, so the same bound applies.

\section{Properness and local parametrizations}\label{sec:covers}

The correspondence records a product $s=zw$ and a common value
$p=A(z)=B(w)$. We first remove the value $m$, where inverse roots can
approach zero. Set
\[
 X=\{(z,w)\in(\Cs)^2:A(z)=B(w)\ne m\},\qquad
 \Omega=\Cs\times(\Chat\setminus\{m\}),
\]
and define
\[
 \Phi:X\longrightarrow\Omega,\qquad
 \Phi(z,w)=(zw,A(z)),\qquad Y=\Phi(X).
\]
For $(s,p)\in\Omega$, the fiber of $\Phi$ is $\mathcal F(s,p)$.
The value $m$ is excluded from $\Omega$; no assumption is made that
$A$ or $B$ omits it.

The map $\Phi$ is proper and has finite fibers. Indeed, on a compact
subset of $\Omega$, the common values stay away from $m$.
Continuity at zero then gives a lower bound $\eta>0$ for both
$|z|$ and $|w|$. If $|zw|\le S$, both coordinates are also bounded
above by $S/\eta$. The inverse image of the compact set is closed
in this compact annulus product, hence compact.
For a fixed value, the inverse roots are discrete, including at
$p=\infty$, where they are poles. Only finitely many lie in the
annulus, and either coordinate determines the other.
Properness also shows that $Y$ is closed in $\Omega$.

We now describe the image near a fixed target $(s_0,p_0)\in\Omega$.
Its central fiber is finite. If the fiber is empty, a neighborhood
of the target misses $Y$. Otherwise choose small neighborhoods of
its inverse pairs. All inverse pairs over sufficiently close targets
lie in those neighborhoods: a contrary sequence would, by properness,
have a subsequence converging to a central pair outside them.

Suppose first that $p_0$ is finite. Near each central inverse pair,
the power coordinates of Section~\ref{sec:small} give
\[
 A(z)=p_0+e_A(z)^{d_A},\qquad
 B(w)=p_0+e_B(w)^{d_B}.
\]
Choose a positive common multiple $N$ of all orders that occur.
For $p=p_0+t^N$, the local inverse pairs are obtained from
\[
 z=e_A^{-1}(\zeta t^{N/d_A}),\qquad
 w=e_B^{-1}(\xi t^{N/d_B}),\qquad
 \zeta^{d_A}=\xi^{d_B}=1.
\]
This gives a finite list of holomorphic pair maps
$P_i(t)=(z_i(t),w_i(t))$. Define their products and target maps by
\begin{equation}\label{eq:raw}
 S_i(t)=z_i(t)w_i(t),\qquad
 \gamma_i(t)=(S_i(t),p(t)),\qquad p(t)=p_0+t^N.
\end{equation}
The parameter disk is taken small enough that all inverse-coordinate
expressions above are defined.
Every map is centered at an inverse pair of $(s_0,p_0)$, so
$S_i(0)=s_0$. For any fixed $N$th root $t$ of $p-p_0$, the list $P_i(t)$
contains every nearby inverse pair at that value. Repetitions in
the list are allowed.

At $p_0=\infty$, use the value coordinate $1/p$. Since all poles
are simple, $1/A$ and $1/B$ are conformal coordinates at the central
inverse roots. The same construction then works with $N=1$ and
$p(t)=1/t$, interpreted as $\infty$ at $t=0$.
For either type of center, write
\[
 \lambda(p)=
 \begin{cases}
 p-p_0,&p_0\ne\infty,\\
 1/p,&p_0=\infty.
 \end{cases}
 \qquad\text{Then}\qquad \lambda(p(t))=t^N.
\]

For clarity, let
$\Delta_\rho=\{t\in\C:|t|<\rho\}$ and
$\Delta_\rho^*=\Delta_\rho\setminus\{0\}$.
We choose the parameter disk and target neighborhood together.
Take a disk $U_s$ about $s_0$ and a value neighborhood whose product
lies within the neighborhood where all inverse pairs are accounted for.
Since $S_i(0)=s_0$, we may choose $\delta>0$ so small that
\[
 S_i(\Delta_\delta)\subset U_s\quad\text{for every }i,\qquad
 W=U_s\times\{p:|\lambda(p)|<\delta^N\}
\]
lies in that neighborhood. Every $\gamma_i(\Delta_\delta)$ lies in
$W$. Conversely, if $(s,p)\in Y\cap W$, choose any parameter root
$t$ with $t^N=\lambda(p)$. Then $|t|<\delta$, and the list $P_i(t)$
contains every inverse pair over $(s,p)$. Consequently,
\begin{equation}\label{eq:cover}
 Y\cap W=\bigcup_i\gamma_i(\Delta_\delta).
\end{equation}
This equality describes the whole local image, including all its
inverse pairs.

\section{Local branches and distinct-pair counts}\label{sec:stabilization}

The parametrizations may repeat the same image or the same inverse
pair. We control these repetitions before counting.
For the finitely many indices $i,j$ and $N$th roots of unity $\omega$,
consider
\[
 S_i(\omega t)-S_j(t).
\]
Each difference either vanishes identically near zero or is nonzero
for all sufficiently small nonzero $t$. Shrink the disk once so that
this holds for all comparisons, and do the same for both coordinates
of $P_i(\omega t)-P_j(t)$. Choose the value radius to be $\delta^N$
again, so that \eqref{eq:cover} remains valid.

Write
\[
 \Gamma_i=\gamma_i(\Delta_\delta),\qquad
 \Gamma_i^*=\gamma_i(\Delta_\delta^*).
\]
The preceding comparisons give four useful consequences.

\begin{enumerate}
\item The punctured images $\Gamma_i^*$ and $\Gamma_j^*$ either coincide
or are disjoint. At an intersection, their parameters differ by an
$N$th root of unity, and the corresponding product comparison must
vanish identically.

\item At every noncentral point, the whole image $Y$ is locally one
holomorphic graph over the value coordinate.
Indeed, $t\mapsto t^N$ has nonzero derivative there. Parametrizations
through the point give the same graph by the first assertion; the
remaining images are locally absent. The closedness assertion below
also rules out their accumulating at that point.

\item Each punctured image has a constant fiber cardinality $K_j$.
For a fixed nonzero parameter, the complete list gives
\begin{equation}\label{eq:exact-fiber}
 \mathcal F(S_j(t),p(t))
 =\{P_i(t):S_i(t)=S_j(t)\}.
\end{equation}
After the shrinking, both the indices satisfying $S_i(t)=S_j(t)$
and the equalities among their pairs $P_i(t)$ are independent of
nonzero $t$. The number of distinct pairs is therefore constant.

\item Every union of the full images $\Gamma_i$, including their
centers, is relatively closed in the value cylinder
$\{(s,p):|\lambda(p)|<\delta^N\}$.
For a convergent sequence in such a union, pass to a fixed index.
Convergence to an interior value gives
$|t_n|\le\delta'<\delta$ eventually, for some $\delta'$.
A subsequence of the parameters then converges inside the disk,
and its image is the required limit.
\end{enumerate}

We also need to compare the central fiber with nearby fibers.
Suppose that, at a finite center, the whole image is a single graph
$s=\psi(p)$. Every product map must then equal
$\psi(p_0+t^N)$, so all punctured image counts $K_j$ have one common
value, say $K$. Choose one index for each distinct central inverse
pair. If two chosen pairs agreed at a sufficiently small nonzero
parameter, their coordinate comparisons would vanish identically,
forcing their central pairs to agree as well. Thus all chosen pairs
remain distinct nearby, and
\begin{equation}\label{eq:centralcount}
 \card\mathcal F(s_0,p_0)\le K.
\end{equation}
The graph condition here concerns the whole image, so that no
additional local branch can contribute extra central pairs.

\section{The locus of large fibers}\label{sec:high-count}

Call a point of $Y$ a \emph{graph point} if, in some product
neighborhood $V_s\times V_p$, the whole image is a single graph
over the common value:
\[
 Y\cap(V_s\times V_p)=\{(\psi(p),p):p\in V_p\},
\]
where $\psi$ is holomorphic in the value coordinate.
At infinity this coordinate is $1/p$.
This definition specifies a graph for the entire local image,
rather than for just one of its parametrizations.

Let $\mathcal T$ be the set of finite-valued graph points whose fibers
contain more than $2k$ inverse pairs. We will prove that
$\mathcal T=\varnothing$ by studying its closure
\[
 H=\overline{\mathcal T}\quad\text{in }\C\times\Chat.
\]
The small-product bound \eqref{eq:small} gives $|s|\ge\varepsilon^2$
on $\mathcal T$, hence also on $H$. Moreover,
$H\cap\Omega\subset Y$, because $Y$ is closed in $\Omega$.

First consider a center in $\Omega$, with the local images from
Section~\ref{sec:stabilization}. Their fiber counts give
\[
 \bigcup_{j:K_j>2k}\Gamma_j^*
 \ \subset\ \mathcal T\cap W
 \ \subset\ \bigcup_{j:K_j>2k}\Gamma_j.
\]
At noncentral points this follows from the graph description and
the constant counts. A finite central point in $\mathcal T$ lies on
a branch with count greater than $2k$ by \eqref{eq:centralcount}.
A center with $p_0=\infty$ is not itself in $\mathcal T$, but its punctured
branches have finite values.

The selected full images are relatively closed, and each is the closure
of its punctured image. Taking closures inside $W$ therefore gives
\begin{equation}\label{eq:interiorH}
 H\cap W=\bigcup_{j:K_j>2k}\Gamma_j.
\end{equation}
Thus every point of $H$ whose value differs from $m$, including a point
at infinity, lies on a full local branch contained in $H$.
The remaining task is to prove the same statement at the excluded
value $m$, where the correspondence need not be proper.

\section{Continuation across the excluded value}\label{sec:omitted}

Fix $(s_0,m)\in H$. We already know $s_0\ne0$.
In a target neighborhood with $|s|\le S$, each point of $\mathcal T$
has more than $2k$ inverse pairs. At most $2k$ have a small coordinate.
At least one pair therefore satisfies
\begin{equation}\label{eq:survivor}
 \varepsilon\le|z|,|w|\le S/\varepsilon.
\end{equation}
The upper bounds follow from $zw=s$.
The choice of this pair may vary with the target.

There are only finitely many roots of $A-m$ and $B-m$ in the compact
annulus in \eqref{eq:survivor}. The local power coordinates at these
roots produce a finite family of target maps
\[
 \beta_j(t)=(b_j(t),m+t^{N_m}),
\]
for a suitable positive common multiple $N_m$ of the root orders.
Each map is obtained from a holomorphic pair of nonzero inverse roots
with common value $m+t^{N_m}$, and $b_j(t)$ is their product.
Every pair satisfying \eqref{eq:survivor} lies in the chosen root
neighborhoods when its common value is sufficiently close to $m$:
otherwise a convergent subsequence would give a root outside those
neighborhoods. Thus every nearby point of $\mathcal T$ lies in the
union of the images of the $\beta_j$.
Unlike the complete family in Section~\ref{sec:covers}, this family
is used only for containment; its indices do not count all inverse pairs.

Retain only the maps with $b_j(0)=s_0$.
To justify this restriction, the finitely many other central products
have a positive minimum distance $d_0$ from $s_0$. Shrink the parameter
disk until $|b_j(t)-b_j(0)|<d_0/3$, and restrict to
$|s-s_0|<d_0/3$. The discarded maps cannot meet this target neighborhood.
If there are no such maps, no restriction is needed.

Apply the finite product comparisons of
Section~\ref{sec:stabilization} to the retained maps. Choose a parameter
radius $\rho_0$ and an open target neighborhood $V_0$ of $(s_0,m)$
such that
\[
 \mathcal T\cap V_0\subset\bigcup_j\beta_j(\Delta_{\rho_0}),
 \qquad V_0\subset\{|p-m|<\rho_0^{N_m}\}.
\]
Since all retained maps have center $(s_0,m)$, choose $0<\rho<\rho_0$
so that their full images on $\Delta_\rho$ lie in $V_0$. Set
\[
 V=V_0\cap\{|p-m|<\rho^{N_m}\},\qquad
 \Lambda_j=\beta_j(\Delta_\rho),\qquad
 \Lambda_j^*=\beta_j(\Delta_\rho^*).
\]
Each $\Lambda_j$ lies in $V$. Conversely, a point of $V$ represented
by a parameter in $\Delta_{\rho_0}$ has
$|t|^{N_m}=|p-m|<\rho^{N_m}$, so that parameter lies in $\Delta_\rho$.
The relative closedness argument of
Section~\ref{sec:stabilization}, applied on the larger disk, gives
\begin{equation}\label{eq:boundarycontainment}
 H\cap V\subset\bigcup_j\Lambda_j,\qquad \Lambda_j\subset V.
\end{equation}
Indeed, a sequence in $\mathcal T$ tending to a point of $V$ is
eventually in $V_0$. Its parameters have a subsequence converging
strictly inside $\Delta_{\rho_0}$, and the limiting value places
that parameter in $\Delta_\rho$.

The punctured images $\Lambda_j^*$ are connected, pairwise coincident
or disjoint, and locally graphs over $p$. They avoid $p=m$.
Their intersections with $H$ are relatively closed. They are also
relatively open. To see this, take $y\in H\cap\Lambda_j^*$.
By \eqref{eq:interiorH}, a full local analytic branch through $y$
lies in $H$. Restrict it to $V$. By
\eqref{eq:boundarycontainment}, it lies in the local graph
$\Lambda_j^*$. Its value coordinate is nonconstant, so the open
mapping theorem shows that it fills a neighborhood of $y$ in that
graph. This proves relative openness.

Connectedness now implies that each $\Lambda_j^*$ either lies wholly
in $H$ or misses $H$. Since $(s_0,m)\in\overline{\mathcal T}$ and
$\mathcal T$ contains no point of value $m$, at least one of these
punctured images meets $\mathcal T$ arbitrarily close to the center.
It lies in $H$, and closedness includes its center.
Hence every point of $H$ over $m$ also lies on a full local branch
contained in $H$.

\section{The product projection and the contradiction}\label{sec:global}

We have shown that every point of $H$ lies on a local branch contained
in $H$. We next show that the product coordinate is nonconstant on
each such branch.

Write one such inverse-pair parametrization as $(z(t),w(t))$.
Suppose, to the contrary, that $z(t)w(t)=s_0\ne0$ identically.
Its common value is nonconstant, so
$z(t)$ is nonconstant as well. At a noncentral parameter with finite
common value, both $g(z(t))$ and $g^\sharp(w(t))$ are nonzero.
The open mapping theorem applied to $z(t)$ gives
\[
 A(z)=B(s_0/z)
\]
on an open set. On that set,
\[
 \frac{zJ_{s_0}'(z)}{J_{s_0}(z)}
 =\frac{zg'(z)}{g(z)}
  -\frac{(s_0/z)(g^\sharp)'(s_0/z)}{g^\sharp(s_0/z)}
 =A(z)-B(s_0/z)=0.
\]
The identity theorem would make $J_{s_0}'$ vanish throughout $\Cs$,
contrary to Lemma~\ref{lem:product}. This argument also applies to
a branch centered at poles, since its noncentral values are finite.

The open mapping theorem therefore makes the product projection
of $H$ open, including at branch centers. The projection is also
closed: $H$ is closed in $\C\times\Chat$, and $\Chat$ is compact.
If $\mathcal T$ were nonempty, this projection would be a nonempty
open and closed subset of the connected plane, hence all of $\C$.
But $|s|\ge\varepsilon^2$ on $H$. The contradiction proves the
fiber bound needed below.

\begin{proposition}\label{prop:empty}
Every finite-valued graph point of $Y$ has at most $2k$ inverse pairs.
Equivalently, $\mathcal T=\varnothing$.
\end{proposition}

\section{A bounded count in every positive interval}\label{sec:radii}

We now return to maximum-modulus points. We choose a radius that avoids
the exceptional targets of the correspondence and at which the radial
maximum has positive derivative.

\begin{proposition}\label{prop:dense}
For every $0<a<b$, there exists $r\in(a,b)$ with $\nu_f(r)\le2k$.
\end{proposition}

\begin{proof}
Let $D$ be the set of finite-valued points of $Y$ that are not graph
points. The local parametrizations show that $D$ is locally finite:
in each sufficiently small target neighborhood, only its center
can fail to be a graph point. A target outside $Y$ has a neighborhood
disjoint from $Y$, since $Y$ is closed in $\Omega$.
The ambient space $\Cs\times(\C\setminus\{m\})$ is second countable,
so $D$ is countable.

Let $Q$ consist of those positive radii $r$ for which $(r^2,p)\in D$
for some finite $p$. The set $Q$ is countable, because each target
determines at most one positive radius. If $r\notin Q$, then for every
real $p>m$, the set $\mathcal F(r^2,p)$ is either empty or a fiber at a graph point.
Properness and Proposition~\ref{prop:empty} give
\[
 \mathcal F(r^2,p)\text{ is finite},\qquad
 \card\mathcal F(r^2,p)\le2k.
\]

Set $G(r)=M_g(r)^2$. The maximum principle makes $G$ strictly increasing:
equality at two radii would put a maximum point from the smaller
circle inside the larger disk and force $g$ to be constant.
Also $G(r)\ge|g(0)|^2>0$.

On the disk $|z|\le b$, the real gradient satisfies
\[
 \bigl\|\nabla |g(z)|^2\bigr\|=2|g(z)||g'(z)|\le C
\]
for some constant $C$. If $z$ maximizes $|g|$ at radius $r\in[a,b]$,
compare it with $(u/r)z$ at radius $u\in[a,b]$.
The two points have distance $|r-u|$, and the joining segment lies
in that disk. Thus $G(r)\le G(u)+C|r-u|$.
Interchanging the radii proves
$|G(r)-G(u)|\le C|r-u|$.

It follows that $G$ is absolutely continuous on $[a,b]$ and
\[
 \int_a^bG'(u)\,du=G(b)-G(a)>0.
\]
Since $G'\ge0$ almost everywhere, the radii where $G$ is differentiable
and $G'>0$ have positive measure in $(a,b)$.
Removing the countable set $Q$ leaves such a radius $r\notin Q$.

Fix any $z\in\mathcal M_g(r)$. Its value is nonzero.
Angular stationarity gives
\[
 \left.\frac{d}{d\theta}|g(e^{i\theta}z)|^2\right|_{\theta=0}
 =-2|g(z)|^2\operatorname{Im}\frac{zg'(z)}{g(z)}=0.
\]
To compute the radial derivative, use
$h_z(u)=|g(uz/r)|^2$. Since $h_z(u)\le G(u)$ and
$h_z(r)=G(r)$, the differentiable function $G-h_z$ has a minimum
at $r$. Therefore
\[
 G'(r)=h_z'(r)
 =\frac{2|g(z)|^2}{r}\operatorname{Re}\frac{zg'(z)}{g(z)}.
\]
This holds for every maximum point, including when several points tie.
Because $|g(z)|^2=G(r)$, all of them have the same real value
\[
 A(z)=p=m+\frac{rG'(r)}{2G(r)}>m.
\]
The injection $z\mapsto(z,\bar z)$ sends $\mathcal M_g(r)$ into
$\mathcal F(r^2,p)$, since
$B(\bar z)=\overline{A(z)}=p$.
The fiber bound gives $\nu_g(r)\le2k$, and
$\mathcal M_f(r)=\mathcal M_g(r)$ completes the proof.
\end{proof}

\section{Local finiteness and sharpness}\label{sec:local-finiteness}

Proposition~\ref{prop:dense} gives a bounded count in every interval.
To turn this into local finiteness of the exceptional radii, we need
a local description of how maximum points move.

\begin{lemma}[One-sided local constancy]\label{lem:onesided}
For every $r_0>0$, there exists $0<\delta<r_0$ such that $\nu_f$ is
constant on each of $(r_0-\delta,r_0)$ and $(r_0,r_0+\delta)$.
The two constants need not agree with each other or with $\nu_f(r_0)$.
\end{lemma}

\begin{proof}
Set $x_0=\log r_0$. There are finitely many central maximum points
by Lemma~\ref{lem:circle-finite}.
Near each such point $z_0$, use a logarithmic coordinate
$\zeta=x+i\theta$, with $e^{\zeta_0}=z_0$.
The function $A(e^\zeta)$ is holomorphic there, since $g(z_0)\ne0$,
and it is nonconstant by the identity theorem.
Its central value $p_*=A(e^{\zeta_0})$ is real by angular stationarity.
The local power coordinate gives
\[
 A(e^\zeta)-p_*=\chi(\zeta)^d,\qquad
 \chi(\zeta_0)=0,
\]
with $\chi$ conformal. Thus the stationary set
$\operatorname{Im}A(e^\zeta)=0$ consists of the real-analytic half-arcs
\begin{equation}\label{eq:stationaryarcs}
 \zeta_j(t)=\chi^{-1}(e^{i\pi j/d}t),\qquad
 t\ge0,\quad j=0,\ldots,2d-1.
\end{equation}

No arc can have constant real part $x_0$. Such an arc would consist
of stationary angular points on the central circle. The squared
modulus would be constant on an open circle arc, hence on the whole
circle by real analyticity, contradicting
Lemma~\ref{lem:circle-finite}.
Consequently, for each arc,
\[
 \operatorname{Re}\zeta_j(t)-x_0
   =\sigma_j t^{\ell_j}u_j(t),\qquad
 \sigma_j\in\{-1,1\},\quad \ell_j\ge1,\quad u_j(0)>0,
\]
where $u_j$ is real analytic.
The change of parameter $t\mapsto t\,u_j(t)^{1/\ell_j}$ has positive
derivative at zero and a real-analytic inverse.
Reparametrize each arc accordingly, retaining the notation $\zeta_j$.
Then
\begin{equation}\label{eq:radialnormalform}
 \operatorname{Re}\zeta_j(t)=x_0+\sigma_jt^{\ell_j}.
\end{equation}
Each shortened half-arc therefore meets each sufficiently close
circle on its side exactly once.

Collect the arcs from all central maxima in one finite list.
Choose a positive common multiple $\ell_*$ of their exponents.
For a fixed side $\sigma\in\{-1,1\}$, put
$x=x_0+\sigma v^{\ell_*}$ with $v>0$.
The arcs with $\sigma_j=\sigma$ give the candidate points
\begin{equation}\label{eq:commonparameter}
 z_j(v)=\exp\!\left(\zeta_j(v^{\ell_*/\ell_j})\right),
 \qquad |z_j(v)|=e^{x_0+\sigma v^{\ell_*}}.
\end{equation}
Each candidate is real analytic in $v$ at zero.
For small positive $v$ they are distinct: half-arcs in a single
chart meet only at their center, the exponential is injective
in that chart, and different central charts have disjoint images
in the $z$-plane.

The squared heights $|g(z_j(v))|^2$ are real analytic.
For any two candidates on the same side, the difference of their squared heights either
vanishes identically or has a nonzero leading Taylor coefficient.
Thus every height comparison has a fixed outcome for sufficiently
small positive $v$. Since there are finitely many comparisons,
one parameter interval works for all of them.

It remains to check that these candidates include all nearby global
maxima. After the preceding shrinkings, choose disjoint angular
windows about the central maxima and a thin radial interval whose
products lie in the final chart neighborhoods.
Compactness gives $\eta>0$ such that
\[
 |g(r_0e^{i\theta})|^2
 \le M_g(r_0)^2-3\eta
 \qquad\text{outside those windows}.
\]
For $r$ close to $r_0$, uniform continuity bounds the squared modulus
outside the windows by $M_g(r_0)^2-2\eta$.
At any fixed central maximizing angle it is greater than
$M_g(r_0)^2-\eta$. Every nearby maximum therefore lies in a window
and on one of the shortened stationary arcs.
By \eqref{eq:radialnormalform}, it is precisely one of the candidates
in \eqref{eq:commonparameter}. Each side has at least one candidate,
since every nearby circle has a maximum.

The indices of the candidates with largest height consequently form a fixed set
on each side. Its cardinality is $\nu_g(e^x)=\nu_f(e^x)$.
Exponentiating the two $x$-intervals and reducing $\delta$ proves
the lemma.
\end{proof}

\begin{proof}[Proof of Theorem~\ref{thm:main}]
Fix $r_0>0$ and use Lemma~\ref{lem:onesided}.
Each punctured interval contains a radius from
Proposition~\ref{prop:dense}. Its constant count must therefore be
at most $2k$. Hence
\[
 E_f\cap(r_0-\delta,r_0+\delta)\subset\{r_0\}.
\]
There is no accumulation point of $E_f$ in $(0,\infty)$.
An infinite intersection with a compact subinterval would have
such an accumulation point, so $E_f$ is locally finite.
Example~\ref{ex:sharp} proves sharpness.
\end{proof}

A locally finite subset of $(0,\infty)$ is countable and has
Lebesgue measure zero. In particular, $\nu_f(r)\le2k$ for almost
every $r>0$.

\begin{example}[Sharpness]\label{ex:sharp}
For an integer $k\ge1$, let
\[
 f_k(z)=\exp\!\left(z^k-\frac{z^{2k}}4\right).
\]
The first coefficient gap is $k$, since
$f_k(z)=1+z^k+O(z^{2k})$.
For $z^k=r^ke^{i\vartheta}$, completing the square gives
\[
 \begin{aligned}
 \log|f_k(z)|
 &=r^k\cos\vartheta-\frac{r^{2k}}4\cos(2\vartheta)\\
 &=\frac{r^{2k}}4+\frac12
   -\frac12(r^k\cos\vartheta-1)^2.
 \end{aligned}
\]
For $0<r\le1$, the maximum occurs only when $\cos\vartheta=1$,
so $z^k=r^k$. For $r>1$, it occurs when $\cos\vartheta=r^{-k}$,
giving exactly the two values
$z^k=1\pm i\sqrt{r^{2k}-1}$.
Each value has $k$ distinct preimages on $|z|=r$. Thus
\[
 \nu_{f_k}(r)=
 \begin{cases}
 k,&0<r\le1,\\
 2k,&r>1.
 \end{cases}
\]
The bound $2k$ is attained on every circle with radius greater than one.
\end{example}

\begin{example}[Exceptional circles are necessary]\label{ex:exception}
Fix an integer $n\ge3$ and $0<\tau<1/8$, and let
\[
 f_{n,\tau}(z)=\exp\!\left(z^n+\tau z(1-z^n)^4\right).
\]
Its value at zero is one and its derivative there is $\tau$, so
its first coefficient gap is one.
On $|z|=1$, one has $|1-z^n|\le2$ and
$1-\operatorname{Re}(z^n)=|1-z^n|^2/2$. Hence
\[
 \begin{aligned}
 1-\operatorname{Re}\bigl(z^n+\tau z(1-z^n)^4\bigr)
 &\ge\frac{|1-z^n|^2}{2}-\tau|1-z^n|^4\\
 &\ge\left(\frac12-4\tau\right)|1-z^n|^2.
 \end{aligned}
\]
The last coefficient is positive. The real part of the exponent is
strictly below one unless $z^n=1$, and equals one at the $n$th roots
of unity. Consequently, $\nu_{f_{n,\tau}}(1)=n>2$.
Exceptional radii cannot therefore be removed from the theorem.
\end{example}

\section{Relation to version 1}\label{sec:scope}

Version~1 \cite{LiV1} already gave the negative answer to the second
question of Erd\H{o}s Problem~\#1117, together with the locally finite
exceptional-set bound and its sharpness.
This version retains those conclusions and expands the continuation
and counting arguments.

The proof uses the closure of the graph points with more than $2k$
inverse pairs. At the excluded value $m$, one bounded inverse pair
places each target in a finite family of parametrized images.
Connectedness then determines which full branches lie in the closure.
This avoids the global generic-degree argument of version~1.
Finally, Proposition~\ref{prop:dense} and Lemma~\ref{lem:onesided}
give local finiteness directly from dense good radii and one-sided
constancy.

\paragraph{Assistance disclosure.}
OpenAI's ChatGPT and Codex (GPT-6) were used to develop the proof and
prepare the manuscript. The human author takes full responsibility for
the mathematical arguments and the final manuscript.

\end{document}